\documentclass[journal]{article}

\usepackage{graphicx,url}
\usepackage{dcolumn}
\usepackage{bm}
\usepackage{amsmath,amsfonts,amssymb,citesort}

\newtheorem{definition}{\noindent{\it Definition}}[section]
\newtheorem{theorem}{\noindent{\it Theorem}}[section]

\newtheorem{remark}[theorem]{\noindent{\it Remark}}
\newtheorem{corollary}[theorem]{\noindent{\it Corollary}}
\newenvironment{proof}{\noindent{\it Proof:}}{$\hfill$ $\Box$\\ }
\newtheorem{example}{\noindent{\it Example}}[section]

\begin{document}

\title{Pseudo Magic Squares}

\author{Giuliano G. La Guardia and Ana Lucia Pereira Baccon
\thanks{The authors are with Department of Mathematics and Statistics,
State University of Ponta Grossa, 84030-900, Ponta Grossa - PR,
Brazil. Corresponding author: Giuliano G. La Guardia ({\tt \small
gguardia@uepg.br}). }}

\maketitle

\begin{abstract}
A magic square of order $n$ is an $n \times n$ square (matrix) whose
entries are distinct nonnegative integers such that the sum of the
numbers of any row and column is the same number, the magic
constant. In this paper we introduce the concept of pseudo magic
squares, \emph{i.e.}, magic squares defined over the ring of
integers, without the restriction of distinct numbers. Additionally,
we generalize this new concept by introducing a group (ring)
structure over it. This new approach can provide useful tools in
order to find new non-isomorphic pseudo magic squares.
\end{abstract}

\section{Introduction}

The concept of magic squares is well known in literature
\cite{Pickover:2002,Andrews:2001,Ahmed:2003}. The Loh-Shu magic
square
\begin{eqnarray*}
\left[
\begin{array}{ccc}
4 & 9 & 2\\
3 & 5 & 7\\
8 & 1 & 6\\
\end{array}
\right]
\end{eqnarray*}
is the oldest known magic square and its invention is attributed to
Fuh-Hi (2858-2738 b.C.) \cite{Pickover:2002}. There exist
interesting papers available in the literature dealing with
constructions of magic squares
\cite{Andrews:2001,Ahmed:2003,Cohen:2003,Ahmed:2004}.  Recently, Xin
\cite{Xin:2008} have constructed all magic squares of order three.
In most works available in literature, the approach adopted is to
apply combinatorial methods to compute more classes of new
(non-isomorphic) magic squares.

In this paper we introduce the concept of pseudo magic square (PMS).
We show that a PMS have a natural group structure. Additionally, we
generalize this new concept in order to obtain a generic magic
square (GMS), which is derived from a arbitrary group (ring). The
structure of a group (ring) induces a group (ring) structure in the
set of GMS's. Based on these facts, one can see that our approach is
quite different of the ones available in literature.

The paper is organized as follows. In Section~\ref{sec2}, we fix the
notation and also define the concept of pseudo magic squares. In
Section~\ref{sec3}, we present the contributions of this paper: some
properties of PMS are shown as well as the introduction of the
concept of generic magic squares with corresponding properties are
exhibited. In Section~\ref{sec4}, a brief summary of this paper is
given.

\section{Preliminaries}\label{sec2}

\emph{Notation}. Throughout this paper, ${\mathbb N}$ denotes the
set of nonnegative integers, ${\mathbb N}-\{0\} = {\mathbb N}^{*}$,
${\mathbb Z}$ is the set of integers. The set of square matrices of
order $n$ with entries in ${\mathbb Z}$ is denoted by ${\mathbb
M}_{n}({\mathbb Z})$; the cardinality of a set $S$ is denoted by
$\mid S\mid$.

\begin{definition}
Let $n \in {\mathbb N}^{*}$. A pseudo magic square of order $n$,
denoted by ${A}_{n}^{\Box}$, is an element of ${\mathbb
M}_{n}({\mathbb Z})$ such that the sum of the numbers of any row and
column is the same number $c_{{A}_{n}^{\Box}}$, the pseudo magic
constant (constant, for short). The set of PMS of order $n$ is
denoted by ${\mathcal P}_{n}^{\Box}$.
\end{definition}
\begin{remark}
To avoid stress of notation, we do not distinguish between an $n
\times n$ matrix or an $n \times n$ square.
\end{remark}
\begin{example}
As an example, we have a PMS of order $4$ given by
\begin{eqnarray*}
\left[
\begin{array}{cccc}
-5 & -5 & -5 & -5\\
-5 & -5 & -5 & -5\\
-5 & -5 & -5 & -5\\
-5 & -5 & -5 & -5\\
\end{array}
\right].
\end{eqnarray*}
\end{example}

\section{The Results}\label{sec3}
In this section we present the results of this work. We start by
showing that the set of PMS endowed with the sum of matrices is a
group:
\begin{theorem}\label{Teo1}
The ordered pair $({\mathcal P}_{n}^{\Box}, +)$ is an abelian group,
where the operation $+$ means the addition of matrices.
\end{theorem}
\begin{proof}
If ${A}_{n}^{\Box}$ and ${B}_{n}^{\Box}$ are PMS of constants
$c_{{A}_{n}^{\Box}}$ and $c_{{B}_{n}^{\Box}}$, respectively, then
${A}_{n}^{\Box}+{B}_{n}^{\Box}$ is a PMS with constant
$c_{{A}_{n}^{\Box}}+ c_{{B}_{n}^{\Box}}$, so ${\mathcal
P}_{n}^{\Box}$ is closed. It is clear that the null PMS of order $n$
is the identity element. The associativity follows trivially of the
associativity of the (additive) group $({\mathbb Z}, +)$. The same
applies to commutativity. It is clear that the inverse of a PMS
${A}_{n}^{\Box}$ of constant $c_{{A}_{n}^{\Box}}$ is the PMS
${[-A]}_{n}^{\Box}$ of constant $-c_{{A}_{n}^{\Box}}$.
\end{proof}

\begin{corollary}\label{Cor1}
$({\mathcal P}_{n}^{\Box}, +)$ is a subgroup of $({\mathbb
M}_{n}({\mathbb Z}), +)$.
\end{corollary}
\begin{proof}
Straightforward.
\end{proof}

It is possible to construct new PMS's from old ones. In fact, the
first method to construct new PMS's is by means of the direct sum
structure (see Theorem~\ref{Teo2}). It is clear that several
structures can be defined over PMS's as well, but such structures
will be considered in other paper.

\begin{theorem}\label{Teo2}
Let ${A}_{n}^{\Box}$ and ${B}_{n}^{\Box}$ be two PMS's. Define the
direct sum ${A}_{n}^{\Box}\oplus {B}_{n}^{\Box}$ as
\begin{eqnarray*}
\left[
\begin{array}{cc}
{A}_{n}^{\Box} & {B}_{n}^{\Box}\\
{B}_{n}^{\Box} & {A}_{n}^{\Box}\\
\end{array}
\right] \in {\mathbb M}_{2n}({\mathbb Z}).
\end{eqnarray*}
Then the set ${\mathcal P}_{n}^{\Box}\oplus{\mathcal P}_{n}^{\Box}$
of direct sums endowed with addition of matrices is an abelian
group.
\end{theorem}
\begin{proof}
The direct sum is well-defined. Let ${A}_{n}^{\Box}$ and
${B}_{n}^{\Box}$ be two PMS's of constants $c_{{A}_{n}^{\Box}}$ and
$c_{{B}_{n}^{\Box}}$, respectively. Then ${A}_{n}^{\Box}\oplus
{B}_{n}^{\Box}$ is a PMS of order $2n$ and constant
$c_{{A}_{n}^{\Box}}+ c_{{B}_{n}^{\Box}}$. If ${C}_{n}^{\Box}$ and
${D}_{n}^{\Box}$ are also PMS's of constants $c_{{C}_{n}^{\Box}}$
and $c_{{D}_{n}^{\Box}}$, respectively, then its sum
$[{A}_{n}^{\Box}\oplus {B}_{n}^{\Box}]+[{C}_{n}^{\Box}\oplus
{D}_{n}^{\Box}]$ is also a PMS of order $2n$ and constant
$c_{{A}_{n}^{\Box}}+c_{{B}_{n}^{\Box}}+c_{{C}_{n}^{\Box}}+c_{{D}_{n}^{\Box}}$.
Then ${\mathcal P}_{n}^{\Box}\oplus{\mathcal P}_{n}^{\Box}$ is
closed. The remaining properties follow similarly as in the proof of
Theorem~\ref{Teo1}.
\end{proof}

\begin{corollary}\label{Cor2}
$({\mathcal P}_{n}^{\Box}\oplus{\mathcal P}_{n}^{\Box}, +)$ is a
subgroup of $({\mathbb M}_{2n}({\mathbb Z}), +)$.
\end{corollary}
\begin{proof}
Straightforward.
\end{proof}

Let ${A}_{n}^{\Box}$ be a PMS of order $n$. We can multiply each
entry of ${A}_{n}^{\Box}$ by an integer $k$. Then the resulting
square is also a PMS square. Similarly, by adding an integer to each
entry we obtain again a PMS. Moreover, one can define the tensor
product of PMS's. Note that such ideas can be explored in order to
generate several properties as well as new structures for PMS's.

Next, we introduce the concept of generic magic squares:
\begin{definition}
Let $(G, \star)$ be an abelian group. A generic magic square, of
order $n$, derived from $(G, \star)$, is an $n\times n$ matrix with
entries in $G$ such that for any row and any column, the result of
the operation among the row(column) elements corresponds to the same
element, the generic magic constant (constant, for short).
\end{definition}
We denote the set of GMS's of order $n$ over $G$ by $[{\mathcal
Gg}]_{G_{n}}^{\Box}$.

\begin{remark}
It is clear that $[{\mathcal Gg}]_{G_{n}}^{\Box}\neq \emptyset$. In
fact, if $e_{G}$ is the identity of $G$, then
$[E]_{G_{n}}^{\Box}=\left[
\begin{array}{cccc}
e_{G} & e_{G} & \cdots & e_{G}\\
e_{G} & e_{G} & \cdots & e_{G}\\
\vdots & \vdots & \cdots & \vdots\\
e_{G} & e_{G} & \cdots & e_{G}\\
\end{array}
\right] \in {\mathbb M}_{n}(G)$ is a GMS over $G$.
\end{remark}

\begin{theorem}\label{Teo3}
The ordered pair $([{\mathcal Gg}]_{G_{n}}^{\Box}, {\star}_{p})$,
where ${\star}_{p}$ denotes the group operation component-wise in
${\mathbb M}_{n}(G)$, is an abelian group.
\end{theorem}
\begin{proof}
It is clear that the set is closed. The identity is the GMS
$[E]_{G_{n}}^{\Box}$. The remaining properties follows similarly as
in the proof of Theorem~\ref{Teo1}. The detailed proof will be
provided in other paper.
\end{proof}
By lack of space, we will not investigate in this paper the
structures of this interesting group, \emph{i.e.}, homomorphisms and
isomorphisms, group actions, properties. However, these subjects of
research will be investigated in a future paper.

Consider now a commutative ring with unit $(R, +, \cdot)$. Similarly
as was done with the group structure, the ring $(R, +, \cdot)$
induces a natural ring structure in the set os GMS as follows:

\begin{definition}
Let $(R, +, \cdot)$ be an commutative ring with unit. A generic
magic square of order $n$, derived from $(R, +, \cdot)$, is an
$n\times n$ matrix with entries in $R$ such that:\\
(1) For any row and any column, the result of the operation $+$
among the row(column) elements corresponds to the same element, the generic
additive magic constant (additive constant, for short);\\
(2) For any row and any column, the result of the operation $\cdot$
among the row(column) elements corresponds to the same element, the
generic multiplicative magic constant (multiplicative constant, for
short).
\end{definition}
We denote by $[{\mathcal Gr}]_{R_{n}}^{\Box}$ the set of GMS's of
order $n$ over $R$. Note that the additive constant can be different
from the multiplicative constant. Analogously to groups, the set
$[{\mathcal Gr}]_{R_{n}}^{\Box}$ is not empty.

Theorem~\ref{Teo4}, given in the following, states that the set
$[{\mathcal Gr}]_{R_{n}}^{\Box}$ endowed with two component-wise
operations form a commutative ring with unit. We give only a sketch
of the proof.

\begin{theorem}\label{Teo4}
The ordered triple $([{\mathcal Gr}]_{R_{n}}^{\Box}, {+}_{p},
{\cdot}_{p})$, where ${+}_{p}$ and ${\cdot}_{p}$ denote the ring
operations component-wise in ${\mathbb M}_{n}(R)$, is a commutative
ring with unit.
\end{theorem}
\begin{proof}
It is easy to see that the set $[{\mathcal Gr}]_{R_{n}}^{\Box}$ is
closed under the operations ${+}_{p}$ and ${\cdot}_{p}$. The
commutativity, associativity and distributivity of such operations
follow from the commutativity, associativity and distributivity,
respectively, of the operations $+$ and $\cdot$ of the ring. The
identity and the unit of $[{\mathcal Gr}]_{R_{n}}^{\Box}$ are
obvious from the context.
\end{proof}

Given a GMS $([{\mathcal Gr}]_{R_{n}}^{\Box}, {+}_{p},
{\cdot}_{p})$, we can operate (the operation considered here is
$\cdot$ of the ring) a fixed element of the ring $r \in R$ with each
entry in $[{\mathcal Gr}]_{R_{n}}^{\Box}$. It is easy to see that
the resulting matrix is also a GMS. In this context, one can derive
more algebraic structures over $([{\mathcal Gr}]_{R_{n}}^{\Box},
{+}_{p}, {\cdot}_{p})$.

\begin{definition}
A matroid $M$ is an ordered pair $(S, \mathcal{I})$ consisting of a
finite set $S$ and a collection $\mathcal{I}$ of
subsets of $S$ satisfying the following three conditions:\\
(I.1) $\emptyset \in \mathcal{I}$;\\
(I.2) If $I \in \mathcal{I}$ and ${I}' \subset I$, then ${I}' \in
\mathcal{I}$;\\
(I.3) If ${I}_1$, ${I}_2$ $\in \mathcal{I}$ and $\mid {I}_1\mid <
\mid {I}_2\mid$, then there exists an element $e \in I_{2} - I_1$
such that $I_1 \cup \{e\} \in \mathcal{I}$.
\end{definition}

It is well known that the concept of matroid generalizes several
algebraic structures. In this context, we describe here, how to
induce a matroid structure on the set $[{\mathcal
Gr}]_{R_{n}}^{\Box}$. In particular, if we consider the class of
vectorial matroid (see \cite{Whitney:1935}), interesting results can
be proved.

\section{Summary}\label{sec4}
In this paper, we have introduced the concept of pseudo magic
squares and generic magic squares. Moreover, we have shown that the
set of generic (pseudo) magic squares has a group (ring) structure.
Additionally, some properties of PMS and GMS as well as
constructions of new PMS's from old ones have also been presented.
From the content exhibited in this paper, it can be observed that
there exist much work with respect to our new approach to be done.
Other interesting possibility is the introduction of a topological
structure in GMS's.

\section*{Acknowledgment}
This research was partially supported by the Brazilian Agencies
CAPES and CNPq.

\end{document}